\documentclass[11pt,twoside]{amsart}

\usepackage[latin1]  {inputenc}%
\usepackage[T1]      {fontenc }%
\usepackage          {amsmath }%
\usepackage          {amsfonts}%
\usepackage          {amssymb }%
\usepackage          {amsthm  }%
\usepackage          {a4wide  }%
\usepackage          {url     }%
\usepackage          {tikz    }%
\usepackage[bookmarks=false,pdfborder={0 0 0.05}]{hyperref}
\usepackage[all]{xy}

\usepackage{enumerate, amsmath, amsfonts, amssymb, amsthm,  wasysym, graphics, graphicx, xcolor, frcursive,comment,bbm}

\usepackage{etex}

\definecolor{darkblue}{rgb}{0.0,0,0.7} % darkblue color
\definecolor{darkred}{rgb}{0.7,0,0} % darkred color
\usepackage{hyperref}
\usepackage[all]{xy}
\usepackage[T1]{fontenc}

\usepackage{array}

\usepackage{float}
\restylefloat{table}

\def\defn#1{{\sf #1}}

\newcommand{\edgedir}{\mathbin{\tikz [semithick, baseline=-0.2ex,-latex, ->] \draw [->] (0pt,0.4ex) -- (1em,0.4ex);}} % edge
\newcommand{\edgedirback}{\mathbin{\tikz [semithick, baseline=-0.2ex,-latex, ->] \draw [<-] (0pt,0.4ex) -- (1em,0.4ex);}} % edge

\newcommand{\ZZ}{\mathbb Z}

\DeclareMathOperator{\NN}{\mathbb{N}}

\newtheorem{theorem}{Theorem}[section]
\newtheorem{corollary}[theorem]{Corollary}
\newtheorem{Proposition}[theorem]{Proposition}
\newtheorem{Lemma}[theorem]{Lemma}

\theoremstyle{definition}
\newtheorem{Definition}[theorem]{Definition}
\newtheorem{remark}[theorem]{Remark}

\title[Non-reduced reflection factorizations of Coxeter elements]{A note on non-reduced reflection factorizations of Coxeter elements}

\author[P.~Wegener]{Patrick Wegener}
\address{Patrick Wegener, Technische Universit\"at Kaiserslautern, Germany}
\email{wegener@mathematik.uni-kl.de}

\author[S.~Yahiatene]{Sophiane Yahiatene}
\address{Sophiane Yahiatene, Universit\"at Bielefeld, Germany}
\email{syahiate@math.uni-bielefeld.de}



\date{\today}

\begin{document}
\newcolumntype{C}[1]{>{\centering\arraybackslash}m{#1}}

\begin{abstract}
We extend a result of Lewis and Reiner from finite Coxeter groups to all Coxeter groups by showing that two reflection factorizations of a Coxeter element lie in the same Hurwitz orbit if and only if they share the same multiset of conjugacy classes. 
\end{abstract}

\maketitle

%\tableofcontents

\section{Introduction}\label{sec:intro}
Given a Coxeter system $(W,S)$ with set of reflections $T$, the braid group (e.g. see \cite{BDSW14} for a definition) acts on reflection factorizations of a given element $w \in W$, that is it acts on tuples $(t_1, \ldots , t_m) \in T^m$ of reflections such that $w=t_1 \cdots t_m$. This action is called \defn{Hurwitz action}. A standard braid group generator $\sigma_i$ (resp. its inverse $\sigma_i^{-1}$) acts by a \defn{Hurwitz move} on a reflection factorization:
\begin{align*}
\sigma_i (t_1, \ldots , t_{i-1}, t_i, t_{i+1}, t_{i+2}, \ldots , t_n) & = (t_1, \ldots , t_{i-1}, t_{i+1}^{t_i}, t_{i}, t_{i+2}, \ldots , t_n),\\
\sigma_i^{-1} (t_1, \ldots , t_{i-1}, t_i, t_{i+1}, t_{i+2}, \ldots , t_n) & = (t_1, \ldots , t_{i-1}, t_{i+1}, t_{i}^{t_{i+1}}, t_{i+2}, \ldots , t_n),
\end{align*}
where we use the notation $g^h:=hgh^{-1}$ for conjugation. 

It has been first observed by Deligne \cite{Del} that this action is transitive on reduced reflection factorizations of a \defn{Coxeter element} if $W$ is finite. The first published proof is due to Bessis \cite[Proposition 1.6.1]{Bes03}. Igusa and Schiffler showed that this statement is true for every Coxeter group \cite[Theorem 1.4]{IS10}. 

The question of how these results extend to non-reduced reflection factorizations has been first addressed by Lewis and Reiner. 

\begin{theorem}
(Lewis-Reiner, \cite[Theorem 1.1]{LR16}) In a finite real reflection group, two reflection factorizations of a Coxeter element lie in the same Hurwitz orbit if and only if they share the same multiset of conjugacy classes.
\end{theorem}

Their proof makes heavy use of a remarkable result for finite Coxeter groups \cite[Corollary 1.4]{LR16}. This result is proved in a case-by-case analysis and seems not to extend to infinite Coxeter groups in general. We prove a similar (but weaker) result for all Coxeter groups (see Lemma \ref{lem:same_refl}). In this way, we obtain that the result of Lewis-Reiner extends to all Coxeter groups. 

\begin{theorem} \label{thm:Main}
Let $(W,S)$ be a Coxeter system of finite rank. Then two reflection factorizations of a Coxeter element in $W$ lie in the same Hurwitz orbit if and only if they share the same multiset of conjugacy classes.
\end{theorem}

\newpage

\section{The proof}
Throughout this note let $(W,S)$ be a \defn{Coxeter system} of finite rank $n \in \NN$ with set of \defn{reflections} $T=\{ wsw^{-1} \mid w \in W,~s \in S\}$. All necessary definitions and facts about Coxeter groups we will use are covered by \cite{Hu90}.

A subgroup $W'$ of $W$ is called \defn{reflection subgroup} if $W'= \langle W' \cap T \rangle$. Each reflection subgroup $W'$ admits a canonical set of generators $\chi(W')$ such that $(W', \chi(W'))$ is a Coxeter system and the set of reflections for $(W', \chi(W'))$ is given by $W' \cap T = \bigcup_{w \in W'} w \chi(W') w^{-1}$ (see \cite[Theorem 1.8]{Dye87}). A reflection subgroup of the form $\langle I \rangle$ for some $I \subseteq S$, is called \defn{parabolic subgroup}.

Let $S =\{ s_1, \ldots , s_n \}$. Then $c= s_{\pi(1)} \cdots s_{\pi(n)}$ is called \defn{Coxeter element} for each permutation $ \pi \in \text{Sym}(n)$. A Coxeter element of a parabolic subgroup is also called \defn{parabolic Coxeter element}.

We denote by $\ell_S$ (resp. $\ell_T$) the length function on $W$ with respect to the generating set $S$ (resp. $T$).

\begin{Definition} \label{def:BruhatGraph}
We define the \defn{Bruhat graph} of $(W,S)$ to be the directed graph $\Omega_{(W,S)}$ on vertex set $W$ and there is a directed edge from $x$ to $y$ if there exists $t \in T$ such that $y=xt$ and $\ell_S(x) < \ell_S(y)$.

Moreover, we denote by $\overline{\Omega}_{(W,S)}$ the corresponding undirected graph and for a subset $X \subseteq W$ we denote by $\Omega_{(W,S)}(X)$ the full subgraph of $\Omega_{(W,S)}$ on vertex set $X$.
\end{Definition}

The following fact is already part of the proof of \cite[Proposition 2.2]{BDSW14}. For sake of completeness we include a proof (which can also be found in the first authors Ph.D. thesis \cite[Proposition 2.3.6]{Weg17}).

\begin{Proposition} \label{prop:BruhatGraphDihedral}
Let $w \in W$ and $t_1,t_2 \in T$ with $t_1 \neq t_2$ such that 
$$ w \edgedir wt_1 \edgedirback wt_1t_2$$
in $\Omega_{(W,S)}$. Then there exist $t_1', t_2 ' \in \langle t_1, t_2 \rangle \cap T$ with $t_1t_2 = t_1' t_2'$ such that one of the following cases hold:
\begin{enumerate}
\item $ w \edgedir wt_1' \edgedir wt_1't_2'= wt_1t_2$
\item $ w \edgedirback wt_1' \edgedirback wt_1't_2'= wt_1t_2$
\item $ w \edgedirback wt_1' \edgedir wt_1't_2'= wt_1t_2$
\end{enumerate}
\end{Proposition}

\begin{proof}
Let $W':= \langle t_1, t_2 \rangle$ and $S':= \chi(W')$. We consider the coset $wW'$ since $w, wt_1, wt_1t_2 \in wW'$. By \cite[Proposition 1.13]{Dye87} we have 
$$\Omega_{(W,S)}(W') \cong \Omega_{(W,S)}(wW') \cong \Omega_{(W',S')},$$
where $(W',S')$ is dihedral and we can check the claim there directly. Inside $W'$ any reflection (element of odd $S'$-length) and any rotation (element of even $S'$-length) are joined by an edge in $\overline{\Omega}_{(W',S')}$ which in $\Omega_{(W',S')}$ is oriented towards the element of greater $S'$-length. For $x \in W'$ there are three possible situations:
\begin{itemize}
\item $\ell_{S'}(x)< \ell_{S'}(xt_1t_2)$
\item $\ell_{S'}(x)> \ell_{S'}(xt_1t_2)$
\item $\ell_{S'}(x)= \ell_{S'}(xt_1t_2)$ (in particular $x \neq e$ since $t_1 \neq t_2$).
\end{itemize}
Hence we can choose $t_1', t_2' \in W' \cap T$ with $t_1't_2' = t_1t_2$ in the three situations such that we have one of the following situations:
\begin{itemize}
\item $x \edgedir xt_1' \edgedir xt_1' t_2'$
\item $x \edgedirback xt_1' \edgedirback xt_1' t_2'$
\item $x \edgedirback xt_1' \edgedir xt_1' t_2'$
\end{itemize}
To see this, note that $x$ and $xt_1t_2$ are both either reflections or rotations. Therefore both are either of odd or even $S'$-length. Thus $\ell_{S'}(x) < \ell_{S'}(xt_1t_2)$ implies $\ell_{S'}(x) +2 \leq \ell_{S'}(xt_1t_2)$ and we find $t_1'$ with $\ell_{S'}(x) < \ell_{S'}(xt_1') < \ell_{S'}(xt_1t_2)$. By setting $t_2' := t_1't_1t_2$ we obtain $x \edgedir xt_1' \edgedir xt_1' t_2'$ and $t_1't_2'=t_1t_2$. The remaining cases are similar.
\end{proof}

\medskip
We use the notation $(t_1, \ldots, t_m) \sim (r_1, \ldots , r_m)$ to indicate that both tuples lie in the same orbit under the Hurwitz action.

\begin{Lemma}\label{lem:same_refl}
Let $w\in W$ with $\ell_{S}(w)=m$ and $w=t_{1}\cdots t_{m+2k}$ with $t_{i}\in T$ for $1\leq i \leq m+2k$ and some $k \in \NN$. Then there exists a braid $\sigma \in \mathcal{B}_{m+2k}$ such that 
$$
\sigma(t_{1},\ldots,t_{m+2k})=(r_{1},\ldots,r_{m},r_{i_1},r_{i_1}, \ldots , r_{i_k},r_{i_k}).
$$
\end{Lemma}

\begin{proof}
We proceed by induction on $k$. Therefore let $k=1$. If there exists a factorization in $\mathcal{B}_{m+2}(t_{1},\ldots,t_{m+2})$ with two identical factors, then we can can shift them to the end of the factorization by just using Hurwitz moves and we are done. Hence let us assume to the contrary that each factorization in $\mathcal{B}_{m+2}(t_{1},\ldots,t_{m+2})$ consists of pairwise different factors. Consider the path of $\Omega_{(W,S)}$ starting in $e$ and ending in $w$ induced by $(t_{1},\ldots,t_{m+2})$. Then Proposition \ref{prop:BruhatGraphDihedral} allows us to replace successively the parts of the path of shape $\star \edgedir \star \edgedirback \star$ by
$$
\star \edgedir \star \edgedir \star, ~\star \edgedirback \star \edgedirback \star, ~\text{or } \star \edgedirback \star \edgedir \star
$$ 
only using the Hurwitz action. The latter is possible since the reflections of the factorizations in the Hurwitz orbit are pairwise different. Since each replacement reduces the sum of the length of the vertices, eventually we get after finitely many replacements a path of the form
\[e\edgedirback t'_{1}\edgedirback t'_{1}t'_{2}\edgedirback \ldots \edgedirback t'_{1}t'_{2}\cdots t'_{p}\edgedir t'_{1}t'_{2}\cdots t'_{p}t'_{p+1}\edgedir \ldots \edgedir t'_{1}\cdots t'_{m+2}=w\]
with $t'_{i}\in T$ for $1\leq i \leq m+2$, that is, the path is first decreasing, then increasing. Since the path starts with $e$, it holds $p=0$ and therefore it has no decreasing part. Altogether it holds that the initial path can be transformed to
\[e \edgedir t'_{1}\edgedir t'_{1}t'_{2}\edgedir \ldots \edgedir t'_{1}\cdots t'_{m+2}=w\]
by using the Hurwitz action. Since the length of $w$ is $m$, the deletion condition yields that $e=wt'_{m+2}\cdots t'_{3}=t'_{1}t'_{2}$, a contradiction to the assumption.

Let $k>1$. If there exists a factorization $(r_1, \ldots , r_{m+k})$ in $\mathcal{B}_{m+2k}(t_1, \ldots , t_{m+2k})$ with $r_i =r_j$ for some $i, j \in \{ 1, \ldots , m+2k \}$ with $i\neq j$, then
\begin{align} \label{equ:HurwitzEquiv}
(t_1, \ldots , t_{m+2k}) \sim (r_1, \ldots , r_{m+k}) \sim (r_1', \ldots, r_{m+2(k-1)}', r_i, r_i)
\end{align}
and we are done by induction. 

Therefore assume again that each factorization in $\mathcal{B}_{m+2k}(t_{1},\ldots,t_{m+2k})$ consists of pairwise different factors. We can argue as before to obtain 
\[e \edgedir t'_{1}\edgedir t'_{1}t'_{2}\edgedir \ldots \edgedir t'_{1}\cdots t'_{m+2k}=w\]
by using the Hurwitz action. In this case the deletion condition yields $e=t_1' \cdots t_{2k}'$. The assertion follows by the following Lemma.
\end{proof}

\begin{Lemma}
If $e = t_1 \cdots t_{2n}$ for some $n \in \NN$ and $t_i \in T$ for all $i \in \{ 1, \ldots , 2n \}$, then there exist reflections $r_1, \ldots , r_n \in T$ such that $(t_1,\ldots t_{2n}) \sim (r_1, r_1, \ldots , r_n , r_n )$. 
\end{Lemma}

\begin{proof}
The assertion is clear for $n=1$. Therefore let $n >1$. If there exists a factorization in $\mathcal{B}_{2n}(t_1,\ldots t_{2n})$ with two identical factors, then we can argue as in the proof of Lemma \ref{lem:same_refl} to obtain a Hurwitz equivalence similar to (\ref{equ:HurwitzEquiv}) and we are done by induction. Therefore assume that each factorization in $\mathcal{B}_{2n}(t_1,\ldots t_{2n})$ has pairwise different factors. Again, with the replacement argument as in the first part of the proof of Lemma \ref{lem:same_refl}, we obtain an increasing path
\[e \edgedir t'_{1}\edgedir t'_{1}t'_{2}\edgedir \ldots \edgedir t'_{1}\cdots t'_{2n}=e\]
from $e$ to $e$ in $\Omega_{(W,S)}$ of length $2n>0$, a contradiction. 
\end{proof}

\begin{remark}
Note that $\ell_T(w) \leq \ell_S(w)$ for all $w \in W$. Therefore the reflection factorization $w = r_1 \cdots r_m$ obtained by Lemma \ref{lem:same_refl} does not have to be a reduced reflection factorization. This is the main difference compared with the key argument \cite[Corollary 1.4]{LR16} in the proof of Lewis-Reiner. However, by \cite[Lemma 2.1]{BDSW14} we have $\ell_S(w)=\ell_T(w)$ for an element $w \in W$ if and only if $w$ is a parabolic Coxeter element. Therefore, if $w$ is a parabolic Coxeter element, then Lemma \ref{lem:same_refl} generalizes \cite[Corollary 1.4]{LR16}. In particular, a reflection factorization of a parabolic Coxeter element can be reduced by just using Hurwitz moves.
\end{remark}

\medskip
A proof of the following fact already implicitely appears in the proof of \cite[Theorem 1.1]{LR16}.

\begin{Lemma} \label{le:Reduction}
Let $t_1, \ldots, t_{n}, t \in T$. Then $(t_1, \ldots, t_{n},t,t) \sim (t_1, \ldots, t_{n}, t^w,t^w)$ for all $w \in \langle t_1, \ldots, t_{n} \rangle$. 
\end{Lemma}

\begin{proof}
Denoting an omitted entry by $\widehat{t_i}$, we obtain
\begin{align*}
(t_1, \ldots, t_{n},t,t) & \sim (t_1, \ldots, t_{i-1},\widehat{t_i}, t_{i+1}^{t_i} \ldots , t_{n}^{t_i}, t^{t_i},t^{t_i}, t_i)\\
{} & \sim (t_1, \ldots, t_{i-1},\widehat{t_i}, t_{i+1}^{t_i} \ldots , t_{n}^{t_i}, t_i, t^{t_i},t^{t_i})\\
{} & \sim (t_1, \ldots,t_{i-1},t_{i}, t_{i+1},\ldots, t_{n},t^{t_i},t^{t_i}).
\end{align*}
\end{proof}

\medskip
\begin{proof}[\textbf{Proof of Theorem \ref{thm:Main}}]
Let $c \in W$ be a Coxeter element and 
$$
c=t_1' \cdots t_{n+2k}' = r_{1}' \cdots r_{n+2k}'
$$
two reflection factorizations of $c$ for some $k \in \ZZ_{\geq 0}$ such that they share the same multiset of conjugacy classes. By Lemma \ref{lem:same_refl} we have 
\begin{align*}
(t_1', \ldots, t_{n+2k}') & \sim (t_1, \ldots , t_n, t_{i_1}, t_{i_1}, \ldots , t_{i_k}, t_{i_k})\\
\text{and }(r_1', \ldots, r_{n+2k}') & \sim (r_1, \ldots , r_n, r_{i_1}, r_{i_1}, \ldots , r_{i_k}, r_{i_k}).
\end{align*}
Since $c=t_1 \cdots t_n=r_1 \cdots r_n$ and $\ell_S(c)=\ell_T(c)=n$ by \cite[Lemma 2.1]{BDSW14}, $(t_1 \ldots t_n)$ and $(r_1 \cdots r_n)$  are reduced reflection factorizations of $c$. Hence we have $(t_1, \ldots, t_n) \sim (r_1, \ldots r_n)$ by \cite[Theorem 1.3]{BDSW14}. In particular $(t_1, \ldots, t_n)$ and $(r_1, \ldots, r_n)$ share the same multiset of conjugacy classes. Hence $t_{i_1}, \ldots, t_{i_k}$ and $r_{i_1}, \ldots, r_{i_k}$ have to share the same multiset of conjugacy classes. Since $(t,t,r,r) \sim (r,r,t,t)$ for all $r,t \in T$, we can assume after a possible renumbering that there exists $w_j \in W$ such that $t_{i_j}^{w_j}=r_{i_j}$ for all $j \in \{1, \ldots, k \}$. We proceed by induction on $k$. As we have seen above, the case $k=0$ is precisely \cite[Theorem 1.3]{BDSW14}. Therefore let $k>0$. By induction we have
\begin{align*}
(t_1, \ldots , t_n, t_{i_1}, t_{i_1}, \ldots,t_{i_{k-1}}, t_{i_{k-1}} , t_{i_k}, t_{i_k}) \sim (r_1, \ldots , r_n, r_{i_1}, r_{i_1}, \ldots,r_{i_{k-1}}, r_{i_{k-1}} , t_{i_k}, t_{i_k})
\end{align*}
As a consequence of \cite[Theorem 1.3]{BDSW14}, we have $W=\langle r_1, \ldots , r_n \rangle$. By what we have pointed out before, there exists $w_k \in \langle r_1, \ldots , r_n \rangle $ such that $t_{i_k}^{w_k} = r_{i_k}$. We conclude
\begin{align*}
(t_1', \ldots, t_{n+2k}') & \stackrel{\phantom{\ref{le:Reduction}}}{\sim} (r_1, \ldots , r_n, r_{i_1}, r_{i_1}, \ldots,r_{i_{k-1}}, r_{i_{k-1}} , t_{i_k}, t_{i_k})\\
{} & \stackrel{\phantom{\ref{le:Reduction}}}{\sim} (r_1, \ldots , r_n,t_{i_k}, t_{i_k}, r_{i_1}, r_{i_1}, \ldots,r_{i_{k-1}}, r_{i_{k-1}})\\
{} & \stackrel{\ref{le:Reduction}}{\sim} (r_1, \ldots , r_n,t_{i_k}^{w_k}, t_{i_k}^{w_k}, r_{i_1}, r_{i_1}, \ldots,r_{i_{k-1}}, r_{i_{k-1}})\\
{} & \stackrel{\phantom{\ref{le:Reduction}}}{=} (r_1, \ldots , r_n,r_{i_k}, r_{i_k}, r_{i_1}, r_{i_1}, \ldots,r_{i_{k-1}}, r_{i_{k-1}})\\
{} & \stackrel{\phantom{\ref{le:Reduction}}}{\sim} (r_1, \ldots , r_n, r_{i_1}, r_{i_1}, \ldots,r_{i_{k-1}}, r_{i_{k-1}},r_{i_k}, r_{i_k})\\
{} & \stackrel{\phantom{\ref{le:Reduction}}}{\sim} (r_1', \ldots, r_{n+2k}').
\end{align*}
\end{proof}

\begin{corollary}
If all edges in the Coxeter graph of $(W,S)$ have odd labels, then two reflection factorizations of the same length of a Coxeter element in $W$ lie in the same Hurwitz orbit.
\end{corollary}

\begin{proof}
If all labels in the Coxeter graph of $(W,S)$ are odd, then all reflections in $T$ are conjugated. Therefore the assertion follows by Theorem \ref{thm:Main}. 
\end{proof}

\end{document}